\documentclass[11pt]{article}

\usepackage[T1]{fontenc}
\usepackage[utf8]{inputenc}
\usepackage[margin=1in]{geometry}
\usepackage{amsmath,amssymb,amsthm,mathtools}
\usepackage{microtype}
\usepackage[bookmarks=false,hidelinks]{hyperref}

\title{Supersaturation for Eventown via Generator Switching}
\usepackage{authblk}

\author[1]{Zicheng Han\thanks{
\href{mailto:hzcqj2020@mail.ustc.edu.cn}{hzcqj2020@mail.ustc.edu.cn}}}

\author[1,2]{Xiande Zhang\thanks{
Corresponding author: \href{mailto:drzhangx@ustc.edu.cn}{drzhangx@ustc.edu.cn}}}

\author[1]{Yuhao Zhao\thanks{
\href{mailto:zhaoyh21@mail.ustc.edu.cn}{zhaoyh21@mail.ustc.edu.cn}}}

\affil[1]{School of Mathematical Sciences, University of Science and Technology of China, Hefei 230026, China}

\affil[2]{Hefei National Laboratory, University of Science and Technology of China, Hefei 230088, China}

\date{}

\newtheorem{theorem}{Theorem}[section]
\newtheorem{lemma}[theorem]{Lemma}
\newtheorem{proposition}[theorem]{Proposition}
\newtheorem{corollary}[theorem]{Corollary}
\newtheorem{conjecture}[theorem]{Conjecture}
\newtheorem{remark}[theorem]{Remark}

\newcommand{\one}{\mathbf 1}
\newcommand{\Span}{\operatorname{span}}

\begin{document}
%\markboth{\shorttitle}{\shorttitle}
\maketitle

\begin{abstract}
An eventown family is a family of even-sized subsets of $[n]$ in which every two distinct members have an even-sized intersection. A classical theorem of Berlekamp and Graver shows that the maximum size of such a family is $2^{\lfloor n/2\rfloor}$. The supersaturation problem for eventown asks how many odd-intersection pairs must occur when this extremal bound is exceeded. For a family $\mathcal F$ of even-sized subsets of $[n]$, let
$e(\mathcal F)$ denote the number of unordered pairs whose intersection size is odd. O'Neill conjectured that if
$|\mathcal F|=2^{\lfloor n/2\rfloor}+s$, then $e(\mathcal F)\ge s\,2^{\lfloor n/2\rfloor-1}$
for
\[
1\le s\le
2^{\lfloor n/2\rfloor}-2^{\lfloor n/4\rfloor}.
\]
Previously, the conjecture was known for $s=1,2$, and, for $s\le 2^{\lfloor n/8\rfloor}/n$ with $n$ sufficiently large. We prove the conjectured bound for
\[
1\le s\le \frac{2^{\lfloor n/2\rfloor}}{26},
\]
extending the known range to a fixed positive proportion of the extremal eventown size. The bound is sharp throughout this range. As further consequences, we derive a lower bound valid for
arbitrary excess $s$, which improves the previously known estimate in an
additional range. We also establish stability and removal results for
families of extremal size satisfying
$e(\mathcal F)<2^{\lfloor n/2\rfloor-1}$, showing that such a family is close to an extremal eventown family and can be made eventown by deleting a small number
of its members.
\end{abstract}



\section{Introduction}
Let \([n]:=\{1,2,\ldots,n\}\). A family \(\mathcal F\subseteq 2^{[n]}\) is
called an \emph{eventown family} if every member of \(\mathcal F\) has even cardinality and every two distinct members have an even-sized intersection.
Berlekamp \cite{Berlekamp} and Graver \cite{Graver} independently proved the fundamental eventown theorem, which asserts that every eventown family \(\mathcal F\) on \([n]\) satisfies \(|\mathcal F|\le 2^{\lfloor n/2\rfloor}\). Moreover, equality can be attained by partitioning \(2\lfloor n/2\rfloor\) elements of \([n]\) into \(\lfloor n/2\rfloor\) disjoint pairs and taking all possible unions of these pairs. This construction gives an eventown family of size \(2^{\lfloor n/2\rfloor}\), showing that this is the exact maximum.  

Once the exact threshold of an extremal problem has been determined, a natural further question is to what extent the defining restriction must be violated above that threshold. This is the \emph{supersaturation problem}: rather than merely asserting the existence of a forbidden configuration, one asks how many such configurations are forced by a prescribed excess.

Supersaturation has become an important theme in extremal set theory. For example, supersaturation versions of the Erd\H{o}s--Ko--Rado theorem concern disjoint pairs appearing above the extremal bound for intersecting
families, while those of Sperner's theorem concern comparable pairs, or
more generally short chains, appearing above the extremal bound for
antichains; see
\cite{BaloghLiuSharifzadehDasTran,BaloghWagner,BollobasLeader,DasGanSudakov1,DasGanSudakov2}.

The eventown problem fits naturally into the same framework. In this setting, a forbidden configuration is a pair of sets whose intersection has odd cardinality. Accordingly, for a family \(\mathcal F\) of even-sized subsets of \([n]\), define
\refstepcounter{footnote}
\[
e(\mathcal F)
=
\left|
\left\{
\{A,B\}\in\binom{\mathcal F}{2}\text{\footnotemark[\value{footnote}]}:
|A\cap B|\equiv 1\pmod 2
\right\}
\right|.
\]
\footnotetext[\value{footnote}]{Here \(\binom{\mathcal F}{2}\) denotes the collection of all two-element subsets of \(\mathcal F\).}
Equivalently, consider the graph whose vertices are all even-sized subsets of \([n]\), with two vertices adjacent whenever the corresponding sets have an odd-sized intersection. Then eventown families are precisely the independent sets in this graph, while the eventown theorem determines its independence number. The supersaturation problem for eventown therefore asks for the minimum possible value of \(e(\mathcal F)\) over all families satisfying \(|\mathcal F|=2^{\lfloor n/2\rfloor}+s\) with \(s\ge 1\).

O'Neill~\cite{ONeill} initiated the systematic study of this problem.
For instance, when \(n=2k=4\ell\), one of his constructions is as follows. For each
\(i\in[\ell]\), let
\(A_{2i-1}=\{4i-3,4i-2\}\), \(A_{2i}=\{4i-1,4i\}\),
\(B_{2i-1}=\{4i-3,4i\}\), and \(B_{2i}=\{4i-2,4i-1\}\). Define
\(\mathcal A=\{\bigcup_{j\in J}A_j:J\subseteq[k]\}\) and
\(\mathcal B=\{\bigcup_{j\in J}B_j:J\subseteq[k]\}\).
Both \(\mathcal A\) and \(\mathcal B\) are extremal eventown families.
Moreover, every \(B\in\mathcal B\setminus\mathcal A\) has odd intersection
with exactly \(2^{k-1}\) members of \(\mathcal A\),
and \(|\mathcal B\setminus\mathcal A|=2^k-2^\ell\). Thus, for every
\(1\le s\le 2^k-2^\ell\), adjoining \(s\) distinct members of
\(\mathcal B\setminus\mathcal A\) to \(\mathcal A\) gives a family of
size \(2^k+s\) with exactly \(s\,2^{k-1}\) odd-intersection pairs.
O'Neill proved that this value is optimal for \(s=1,2\) and proposed
the following conjecture.

\begin{conjecture}[O'Neill~\cite{ONeill}]\label{conj:oneill}
Let \(\mathcal F\) be a family of
\(2^{\lfloor n/2\rfloor}+s\) even-sized subsets of \([n]\). If
\(1\le s\le
2^{\lfloor n/2\rfloor}-2^{\lfloor n/4\rfloor}\), then
\(e(\mathcal F)\ge s\,2^{\lfloor n/2\rfloor-1}\).
\end{conjecture}

Several subsequent works have approached this conjecture from complementary directions.  Antipov and Cherkashin \cite{AntipovCherkashin} applied the
Lov\'asz theta approach and spectral methods to the odd-intersection graph to obtain half the conjectured bound for even \(n\).  Wei, Zhao, Zhang and Ge \cite{WeiZhaoZhangGe} used extremal graph theory and discrete Fourier analysis to study this problem. They proved that, for
all sufficiently large \(n\), Conjecture~\ref{conj:oneill} holds whenever \(s\le \frac{2^{\lfloor n/8\rfloor}}{n}\), and they also established a general lower bound of essentially half the conjectured value for arbitrary \(n\) and \(s\).  Hence Conjecture~\ref{conj:oneill} has been verified for a growing range of the excess, although that range is still exponentially smaller than the extremal eventown size
\(2^{\lfloor n/2\rfloor}\).

On the construction side, Niu, Hang and Cao \cite{NiuHangCao} recently showed
that the conjectured lower bound is attained in a much wider range: for
every \(1\le s\le 2^{\lfloor n/2\rfloor}-2\), there exists a family of \(2^{\lfloor n/2\rfloor}+s\) even-sized subsets with exactly \(s\,2^{\lfloor n/2\rfloor-1}\) odd-intersection pairs. They also obtained additional constructions attaining equality based on symmetric designs.  These results show that the lower bound predicted by O'Neill is best possible far beyond the range in which it was originally formulated.  The main remaining difficulty is therefore to prove the universal lower bound for substantially larger values of \(s\).

Our main result establishes the conjectured bound for a fixed positive
proportion of the extremal eventown size.

\begin{theorem}\label{thm:main}
Let \(\mathcal F\) be a family of
\(2^{\lfloor n/2\rfloor}+s\) even-sized subsets of \([n]\).  If
\(1\le s\le \frac{2^{\lfloor n/2\rfloor}}{26}\), then
\(e(\mathcal F)\ge s\,2^{\lfloor n/2\rfloor-1}\).
\end{theorem}

In particular, Theorem~\ref{thm:main} enlarges the previously known exact
range from \(2^{\lfloor n/8\rfloor}/n\) to a constant proportion of
\(2^{\lfloor n/2\rfloor}\).

Our proof is based on a vector-space model over \(\mathbb F_2\). We choose
a maximum eventown subfamily and embed it into a generator, namely, a
maximum-dimensional subspace whose vectors are pairwise orthogonal with
respect to the standard bilinear form over \(\mathbb F_2\). We then combine
switching arguments with Fourier analysis to estimate the number of
odd-intersection pairs between the chosen subfamily and the remaining sets.
A greedy decomposition of the whole family into eventown layers then yields
the desired lower bound. The same approach also gives an improved lower
bound for general \(s\), together with stability and removal results for
families with few odd-intersection pairs.

\paragraph{Organization.} The rest of this paper is organized as follows. In Section~\ref{sec:model}, we introduce the vector-space model and give some basic properties of generators.
In Section~\ref{sec:switching}, we prove the key switching--Fourier estimate. We then prove Theorem~\ref{thm:main} and derive the corresponding averaging result in Section~\ref{sec:main-proof}. Finally, Section~\ref{sec:stability} establishes the stability and removal results.

\section{The vector-space model and generators}\label{sec:model}

Identify a subset of \([n]\) with its characteristic vector in
\(\mathbb F_2^n\), and use the standard bilinear form
\[
 \langle x,y\rangle=\sum_{i=1}^n x_i y_i\pmod 2.
\]
Write $\one=(1,\dots,1)$ and set
\[
 \mathcal V=\{x\in\mathbb F_2^n:\langle x,\one\rangle=0\}.
\]
The vectors in \(\mathcal V\) are precisely the characteristic vectors of
even-sized subsets.  Since
\(\langle x,x\rangle=\langle x,\one\rangle=0\) for every
\(x\in\mathcal V\), the restriction of the bilinear form to
\(\mathcal V\) is alternating.  Two vectors \(x\ne y\) in \(\mathcal V\) are adjacent in the
\emph{odd-intersection graph} if \(\langle x,y\rangle=1\), where the graph has the vertex set \(\mathcal V\).

For \(\mathcal A\subseteq\mathcal V\), let \(e(\mathcal A)\) denote the
number of edges induced by \(\mathcal A\).  For disjoint
\(\mathcal A,\mathcal B\subseteq\mathcal V\), let
\[
 e(\mathcal A,\mathcal B)
 =
 \bigl|\{\{x,y\}:x\in\mathcal A,\ y\in\mathcal B,
          \langle x,y\rangle=1\}\bigr|.
\]
Thus \(e(\mathcal A,\mathcal B)\) is the number of cross-edges between
\(\mathcal A\) and \(\mathcal B\).

A subspace \(\mathcal W\le\mathcal V\) is \emph{totally isotropic} if \(\langle w,w'\rangle=0\) for all \(w,w'\in\mathcal W\).  A totally isotropic subspace is \emph{maximal} if it is not properly contained in another totally isotropic subspace.  It is \emph{maximum-dimensional} if it has the maximum dimension.  We call every maximum-dimensional totally isotropic subspace of \(\mathcal V\) a \emph{generator}.

\begin{lemma}[Structure and extension]\label{lem:structure}
Every totally isotropic subspace of \(\mathcal V\) is contained in a
generator.  Every generator has dimension \(\lfloor n/2\rfloor\), and hence
has \(2^{\lfloor n/2\rfloor}\) elements.  Moreover, for every generator
\(\mathcal W\),
\[
 \mathcal W^\perp\cap\mathcal V=\mathcal W,
\]
where \(\mathcal W^\perp\) is taken in the full space
\(\mathbb F_2^n\).
\end{lemma}

\begin{proof}
We first recall a standard fact about a nondegenerate alternating space of
dimension \(2\ell\).  If \(U\) is totally isotropic, then
\(U\subseteq U^\perp\), so
\[
 2\dim U\le \dim U+\dim U^\perp=2\ell.
\]
Thus \(\dim U\le\ell\).  If \(\dim U<\ell\), then
\(U^\perp\) properly contains \(U\); choosing
\(v\in U^\perp\setminus U\) enlarges \(U\) to the totally isotropic
subspace \(U+\langle v\rangle\).  Repeating this step extends \(U\) to a
totally isotropic subspace of dimension \(\ell\).

Suppose first that \(n=2\ell+1\).  Then \(\mathcal V=\one^\perp\) has
dimension \(2\ell\).  Its radical is
\[
 \mathcal V\cap\mathcal V^\perp
 =
 \mathcal V\cap\langle\one\rangle
 =\{0\},
\]
because \(\one\notin\mathcal V\) when \(n\) is odd.  Hence the restricted
form is nondegenerate, and the preceding argument shows that every totally
isotropic subspace extends to one of dimension
\(\ell=\lfloor n/2\rfloor\).  If \(\mathcal W\) has that dimension, then
its orthogonal complement inside \(\mathcal V\) also has dimension
\(\ell\), and therefore equals \(\mathcal W\).  This is precisely
\(\mathcal W^\perp\cap\mathcal V=\mathcal W\).

Now suppose that \(n=2\ell\).  The radical of the restricted form on
\(\mathcal V\) is \(\langle\one\rangle\).  The quotient \(\mathcal V/\langle\one\rangle\) is a nondegenerate alternating space of dimension \(2\ell-2\).  Given a totally isotropic subspace \(U\le\mathcal V\), the subspace
\(U+\langle\one\rangle\) is totally isotropic, and its image in the quotient extends to a totally isotropic subspace of dimension \(\ell-1\).  The full
preimage of that extension is a totally isotropic subspace of \(\mathcal V\) of dimension \(\ell=\lfloor n/2\rfloor\) containing
\(U\).  No totally isotropic subspace can have larger dimension, because its
image after adjoining the radical is isotropic in the quotient.

Finally, if \(\mathcal W\) is a generator in the even case, then
\(\dim\mathcal W=\ell\).  Since the standard form on
\(\mathbb F_2^{2\ell}\) is nondegenerate,
\(\dim\mathcal W^\perp=\ell\).  Total isotropy gives
\(\mathcal W\subseteq\mathcal W^\perp\), and equality follows.  In
particular, \(\mathcal W^\perp\cap\mathcal V=\mathcal W\).
\end{proof}

\begin{corollary}\label{cor:half-neighbors}
If \(\mathcal W\) is a generator and
\(x\in\mathcal V\setminus\mathcal W\), then
\[
 \bigl|\{w\in\mathcal W:\langle x,w\rangle=1\}\bigr|
 =2^{\lfloor n/2\rfloor-1}.
\]
\end{corollary}

\begin{proof}
By Lemma~\ref{lem:structure}, the functional
\(w\mapsto\langle x,w\rangle\) on \(\mathcal W\) is nonzero.  Every
nonzero linear functional on a finite vector space over \(\mathbb F_2\)
takes the value \(1\) on exactly half of the elements.
\end{proof}

\begin{lemma}[Maximum eventown subfamilies]\label{lem:maximum-eventown}
Every eventown subfamily of \(\mathcal V\) is contained in a generator.  More
precisely, let \(\mathcal R\subseteq\mathcal V\), let \(\mathcal E\) be a
maximum eventown subfamily of \(\mathcal R\), and let \(\mathcal W\) be any
generator containing \(\mathcal E\).  Then
\[
 \mathcal R\cap\mathcal W=\mathcal E,
\]
and \(\mathcal W\) maximizes \(|\mathcal R\cap\mathcal W'|\) over all
generators \(\mathcal W'\).  Conversely, if \(\mathcal W\) is a generator
maximizing \(|\mathcal R\cap\mathcal W|\), then
\(\mathcal R\cap\mathcal W\) is a maximum eventown subfamily of
\(\mathcal R\).
\end{lemma}

\begin{proof}
If \(\mathcal E\) is eventown, then every two vectors in
\(\mathcal E\), including equal vectors, have inner product \(0\).  By
bilinearity, \(\Span(\mathcal E)\) is totally isotropic.  Lemma
\ref{lem:structure} therefore embeds \(\mathcal E\) in a generator.

Now suppose that \(\mathcal E\) is maximum in \(\mathcal R\) and that
\(\mathcal W\) contains \(\mathcal E\).  The family
\(\mathcal R\cap\mathcal W\) is eventown and contains \(\mathcal E\), so
the maximum cardinality of \(\mathcal E\) gives
\(\mathcal R\cap\mathcal W=\mathcal E\).  For any other generator
\(\mathcal W'\), the family \(\mathcal R\cap\mathcal W'\) is eventown;
therefore
\[
 |\mathcal R\cap\mathcal W'|
 \le |\mathcal E|
 =|\mathcal R\cap\mathcal W|.
\]
Thus \(\mathcal W\) is a maximum-intersection generator for
\(\mathcal R\).

Conversely, suppose that \(\mathcal W\) maximizes the intersection with
\(\mathcal R\).  The family \(\mathcal R\cap\mathcal W\) is eventown.  If
\(\mathcal E'\subseteq\mathcal R\) is any eventown family, embed it in a
generator \(\mathcal W'\).  Then
\[
 |\mathcal E'|
 \le |\mathcal R\cap\mathcal W'|
 \le |\mathcal R\cap\mathcal W|.
\]
Hence \(\mathcal R\cap\mathcal W\) has maximum cardinality among eventown
subfamilies of \(\mathcal R\).
\end{proof}

\begin{remark}[Sharpness]\label{rem:sharpness}
The lower bound in Theorem~\ref{thm:main} is best possible. Indeed,
the constructions of O'Neill~\cite{ONeill} and, more generally, those of
Niu, Hang and Cao~\cite{NiuHangCao} provide families \(\mathcal F\) of
\(2^{\lfloor n/2\rfloor}+s\) even-sized subsets satisfying
\[
e(\mathcal F)=s\,2^{\lfloor n/2\rfloor-1}
\]
throughout a range containing the one in Theorem~\ref{thm:main}.
Since these constructions are not used in our proof, we do not repeat
their details here.
\end{remark}

\section{The switching--Fourier estimate}\label{sec:switching}

The next lemma is the main input to the proof of
Theorem~\ref{thm:main}.  Its formulation applies directly to every remainder
that occurs in the greedy decomposition.

\begin{lemma}[Switching--Fourier estimate]\label{lem:switching-fourier}
Let \(\mathcal R\subseteq\mathcal V\) be a nonempty subset. Let \(\mathcal E\) be a
maximum eventown subfamily of \(\mathcal R\), and write \(\mathcal Y=\mathcal R\setminus\mathcal E\). 
Then
\begin{equation}
 \bigl(|\mathcal E||\mathcal Y|-2e(\mathcal E,\mathcal Y)\bigr)^2
 \le
 2^{\lfloor n/2\rfloor+1}|\mathcal E|e(\mathcal E,\mathcal Y).
\label{eq:switching-fourier}
\end{equation}
Consequently, we have
\begin{equation}
 |\mathcal Y|
 \le
 \frac{2e(\mathcal E,\mathcal Y)}{|\mathcal E|}
 +
 \sqrt{
       \frac{2^{\lfloor n/2\rfloor+1}e(\mathcal E,\mathcal Y)}{|\mathcal E|}
      }.
\label{eq:remainder-bound}
\end{equation}
\end{lemma}

\begin{proof}
By Lemma~\ref{lem:maximum-eventown}, choose a generator \(\mathcal W\)
containing \(\mathcal E\).  Then
\(\mathcal R\cap\mathcal W=\mathcal E, \mathcal Y=\mathcal R\setminus\mathcal W\), 
and \(\mathcal W\) maximizes \(|\mathcal R\cap\mathcal W'|\) over all
generators \(\mathcal W'\).

Partition \(\mathcal Y\) by the cosets of \(\mathcal W\) in
\(\mathcal V\) other than \(\mathcal W\) itself.  For an occupied coset
\(\Gamma=x+\mathcal W\), define
\[
 \lambda_\Gamma:\mathcal W\longrightarrow\mathbb F_2,
 \qquad
 \lambda_\Gamma(w)=\langle x,w\rangle.
\]
This definition is independent of the representative \(x\). Indeed, replacing
\(x\) by \(x+w_0\), with \(w_0\in\mathcal W\), does not change the value
because \(\mathcal W\) is totally isotropic.  The functional is nonzero,
for otherwise \(x\in\mathcal W^\perp\cap\mathcal V=\mathcal W\), contrary
to \(\Gamma\ne\mathcal W\).  Moreover, distinct cosets induce distinct functionals:
if \(x+\mathcal W\) and \(x'+\mathcal W\) induce the same functional, then
\(x-x'\in\mathcal W^\perp\cap\mathcal V=\mathcal W\), so the cosets are
equal.

Set \(K_\Gamma=\ker\lambda_\Gamma,
 m_\Gamma=|\mathcal Y\cap\Gamma|\), and \(b_\Gamma =\bigl|\{y\in\mathcal E:\lambda_\Gamma(y)=1\}\bigr|\).
Since \(\lambda_\Gamma\) is nonzero, \(K_\Gamma\) is a hyperplane in
\(\mathcal W\), and \(\Gamma\) is the disjoint union of two affine
\(K_\Gamma\)-cosets.  We call these two cosets the affine halves of
\(\Gamma\).

Every vector in \(\Gamma\) has exactly \(b_\Gamma\) neighbors in
\(\mathcal E\).  Indeed, if \(z,z'\in\Gamma\), then
\(z'-z\in\mathcal W\), and for every \(y\in\mathcal E\subseteq\mathcal W\),
\[
 \langle z',y\rangle
 =
 \langle z,y\rangle+\langle z'-z,y\rangle
 =
 \langle z,y\rangle.
\]
It follows that
\begin{equation}
 e(\mathcal E,\mathcal Y)=\sum_\Gamma m_\Gamma b_\Gamma.
\label{eq:q-cosets}
\end{equation}

We next obtain the switching bound that controls the weights
\(m_\Gamma\).  Let \(H\) be one of the two affine halves of an occupied
coset \(\Gamma\).  If \(\mathcal Y\cap H=\varnothing\), there is nothing to
prove.  Otherwise choose \(z\in\mathcal Y\cap H\).  For every
\(w\in\mathcal W\), the equality
\(\langle z,w\rangle=\lambda_\Gamma(w)\) holds, and hence
\[
 \mathcal K:=\mathcal W\cap z^\perp=K_\Gamma.
\]
Define \(\mathcal L=\mathcal K\oplus\langle z\rangle\). The space \(\mathcal K\) is totally isotropic, \(z\) is orthogonal to
\(\mathcal K\), and \(\langle z,z\rangle=0\) because
\(z\in\mathcal V\).  Also \(z\notin\mathcal W\), so
\(\dim\mathcal L=\dim\mathcal W\).  Thus \(\mathcal L\) is a generator.
Moreover, we have \(\mathcal L\cap\mathcal W=\mathcal K\) and
\[
 \mathcal L\setminus\mathcal W=z+\mathcal K=H.
\]
Since \(\mathcal W\) maximizes \(|\mathcal R\cap\mathcal W'|\) over all
generators \(\mathcal W'\), we have
\[
 |\mathcal R\cap\mathcal L|
 \le
 |\mathcal R\cap\mathcal W|
 =|\mathcal E|.
\]
Using the displayed decomposition of \(\mathcal L\), we have the exact
identity
\[
 \mathcal R\cap\mathcal L
 =
 (\mathcal E\cap\mathcal K)
 \sqcup
 (\mathcal Y\cap H).
\]
Since \(|\mathcal E\cap\mathcal K|=|\mathcal E|-b_\Gamma\), it follows
that
\[
 |\mathcal Y\cap H|\le b_\Gamma.
\]
The two affine halves cover \(\Gamma\), so \(m_\Gamma\le 2b_\Gamma\) and hence
\begin{equation}
 \sum_\Gamma m_\Gamma^2
 \le
 2\sum_\Gamma m_\Gamma b_\Gamma
 =2e(\mathcal E,\mathcal Y).
\label{eq:weight-energy}
\end{equation}

We now turn to the Fourier step.  The quantity on the left side of
\eqref{eq:switching-fourier} can be written as a weighted sum of Fourier
coefficients.  Let \(\chi_{\mathcal E}\) denote the indicator function of
\(\mathcal E\) on \(\mathcal W\).  For \(\lambda\in\mathcal W^*\), use the
unnormalized coefficient
\[
 \widehat{\chi_{\mathcal E}}(\lambda)
 =
 \sum_{w\in\mathcal W}
 \chi_{\mathcal E}(w)(-1)^{\lambda(w)}.
\]
For an occupied coset \(\Gamma\),
\[
 \widehat{\chi_{\mathcal E}}(\lambda_\Gamma)
 =
 |\mathcal E|-2b_\Gamma.
\]
Parseval's identity in this normalization gives
\[
 \sum_{\lambda\in\mathcal W^*}
 \widehat{\chi_{\mathcal E}}(\lambda)^2
 =
 |\mathcal W|
 \sum_{w\in\mathcal W}\chi_{\mathcal E}(w)^2
 =
 2^{\lfloor n/2\rfloor}|\mathcal E|.
\]
Since distinct occupied cosets give distinct functionals, we deduce
\begin{equation}
 \sum_\Gamma (|\mathcal E|-2b_\Gamma)^2
 \le
 2^{\lfloor n/2\rfloor}|\mathcal E|.
\label{eq:fourier-energy}
\end{equation}
Finally, by \eqref{eq:q-cosets},
\[
 |\mathcal E||\mathcal Y|-2e(\mathcal E,\mathcal Y)
 =
 \sum_\Gamma m_\Gamma(|\mathcal E|-2b_\Gamma).
\]
The Cauchy--Schwarz inequality, followed by \eqref{eq:weight-energy} and
\eqref{eq:fourier-energy}, yields
\[
 \bigl(|\mathcal E||\mathcal Y|-2e(\mathcal E,\mathcal Y)\bigr)^2
 \le
 \left(\sum_\Gamma m_\Gamma^2\right)
 \left(\sum_\Gamma(|\mathcal E|-2b_\Gamma)^2\right)
 \le
 2^{\lfloor n/2\rfloor+1}|\mathcal E|e(\mathcal E,\mathcal Y).
\]
This proves \eqref{eq:switching-fourier}.  Taking square roots gives
\[
 |\mathcal E||\mathcal Y|-2e(\mathcal E,\mathcal Y)
 \le
 \bigl||\mathcal E||\mathcal Y|-2e(\mathcal E,\mathcal Y)\bigr|
 \le
 \sqrt{2^{\lfloor n/2\rfloor+1}|\mathcal E|e(\mathcal E,\mathcal Y)}.
\]
Adding \(2e(\mathcal E,\mathcal Y)\) and dividing by \(|\mathcal E|\) gives
\eqref{eq:remainder-bound}.
\end{proof}

\begin{corollary}[Large remainder]\label{cor:large-remainder}
Under the assumptions in Lemma~\ref{lem:switching-fourier}, one has
\[
 |\mathcal Y|>2^{\lfloor n/2\rfloor-1}
 \quad\Longrightarrow\quad
 e(\mathcal E,\mathcal Y)
 >
 2^{\lfloor n/2\rfloor-4}|\mathcal E|.
\]
Moreover,
\[
 |\mathcal Y|>\frac58\,2^{\lfloor n/2\rfloor}
 \quad\Longrightarrow\quad
 e(\mathcal E,\mathcal Y)
 >
 \frac{3}{32}\,2^{\lfloor n/2\rfloor}|\mathcal E|.
\]
\end{corollary}

\begin{proof}
For the first assertion, suppose that
\(e(\mathcal E,\mathcal Y) \le 2^{\lfloor n/2\rfloor-4}|\mathcal E|\).  Then
\eqref{eq:remainder-bound} gives
\[
 |\mathcal Y|
 \le
 2^{\lfloor n/2\rfloor-3}
 +
 2^{\lfloor n/2\rfloor-3/2}
 <
 2^{\lfloor n/2\rfloor-1}.
\]
This contradicts the hypothesis of the first assertion.
For the second assertion, suppose instead that
\[
 e(\mathcal E,\mathcal Y)
 \le
 \frac{3}{32}\,2^{\lfloor n/2\rfloor}|\mathcal E|.
\]
Then \eqref{eq:remainder-bound} yields
\[
 |\mathcal Y|
 \le
 \left(\frac{3}{16}+\frac{\sqrt3}{4}\right)2^{\lfloor n/2\rfloor}
 =
 \frac{3+4\sqrt3}{16}\,2^{\lfloor n/2\rfloor}
 <
 \frac58\,2^{\lfloor n/2\rfloor},
\]
a contradiction.
\end{proof}

\begin{lemma}[Small remainder]\label{lem:small-remainder}
Let \(\mathcal F\subseteq\mathcal V\) be a subset of size \(|\mathcal F|=2^{\lfloor n/2\rfloor}+s\), and let \(\mathcal E\) be a maximum eventown subfamily of \(\mathcal F\).
If \(|\mathcal F\setminus\mathcal E| \le 2^{\lfloor n/2\rfloor-1}\), then
\[
 e(\mathcal F)\ge s\,2^{\lfloor n/2\rfloor-1}.
\]
\end{lemma}

\begin{proof}
Choose a generator \(\mathcal W\) containing \(\mathcal E\).  By Lemma
\ref{lem:maximum-eventown},
\(\mathcal F\cap\mathcal W=\mathcal E\).  Set
\(\mathcal X=\mathcal F\setminus\mathcal W=\mathcal F\setminus\mathcal E\) and
\(\mathcal T=\mathcal W\setminus\mathcal F\).
Then \(|\mathcal X|\le 2^{\lfloor n/2\rfloor-1}\) by hypothesis. Moreover,
\[
 |\mathcal X|-|\mathcal T|=|\mathcal F|-|\mathcal W|=(2^{\lfloor n/2\rfloor}+s)-2^{\lfloor n/2\rfloor}=s.
\]
Note that every vector in \(\mathcal X\) has exactly
\(2^{\lfloor n/2\rfloor-1}\) neighbors in \(\mathcal W\); hence
\[
 e(\mathcal X,\mathcal E)
 =e(\mathcal X,\mathcal W)-e(\mathcal X,\mathcal T)=|\mathcal X|\,2^{\lfloor n/2\rfloor-1}
 -e(\mathcal X,\mathcal T).
\]
Since
\[
 e(\mathcal X,\mathcal T)
 \le
 |\mathcal X||\mathcal T|
 \le
 2^{\lfloor n/2\rfloor-1}|\mathcal T|,
\]
we obtain
\[
 e(\mathcal F)
 \ge
 e(\mathcal X,\mathcal E)
 =|\mathcal X|\,2^{\lfloor n/2\rfloor-1}
 -e(\mathcal X,\mathcal T)
 \ge
 (|\mathcal X|-|\mathcal T|)
 2^{\lfloor n/2\rfloor-1}
 =
 s\,2^{\lfloor n/2\rfloor-1},
\]
completing the proof of the lemma.
\end{proof}

\section{Proof of the main theorem and consequences}\label{sec:main-proof}

\begin{proof}[Proof of Theorem~\ref{thm:main}]
Let \(1\le s\le \frac{1}{26}\,2^{\lfloor n/2\rfloor}\). Assume for a contradiction that \(e(\mathcal F)<s\,2^{\lfloor n/2\rfloor-1}\).

We greedily decompose \(\mathcal F\) into eventown layers.  Set
\(\mathcal R_0=\mathcal F\).  Whenever \(\mathcal R_{i-1}\ne\varnothing\),
choose a maximum eventown subfamily
\(\mathcal E_i\subseteq\mathcal R_{i-1}\), and define
\[
 \mathcal R_i=\mathcal R_{i-1}\setminus\mathcal E_i,
\]
and let \(a_i=|\mathcal E_i|\).
The process ends when some remainder is empty.  Since the remainders are
nested, every eventown subfamily of \(\mathcal R_i\) is also an eventown
subfamily of \(\mathcal R_{i-1}\).  Therefore
\begin{equation}
 a_1\ge a_2\ge\cdots.
\label{eq:layer-monotonicity}
\end{equation}

The cross-edge sets between each layer \(\mathcal E_i\) and the corresponding
remainder \(\mathcal R_i\) form an exact partition of the edges of
\(\mathcal F\).  Indeed, each \(\mathcal E_i\) is eventown and hence has
no internal edges.  If the endpoints of an edge lie in distinct layers, then
the endpoint in the earlier layer \(\mathcal E_i\) is paired with an endpoint
that still belongs to \(\mathcal R_i\).  The edge is counted there and in no
other term.  Consequently,
\begin{equation}
 e(\mathcal F)
 =
 \sum_i e(\mathcal E_i,\mathcal R_i).
\label{eq:edge-partition}
\end{equation}

If
\(|\mathcal R_1|\le2^{\lfloor n/2\rfloor-1}\), then Lemma
\ref{lem:small-remainder}, applied to the maximum eventown subfamily
\(\mathcal E_1\), gives the required lower bound, contrary to our
assumption.  Hence
\begin{equation}
 |\mathcal R_1|>2^{\lfloor n/2\rfloor-1}.
\label{eq:first-remainder-large}
\end{equation}

We first dispose of the intermediate range.  If
\[
 |\mathcal R_1|\le \frac58\,2^{\lfloor n/2\rfloor},
\]
then the first assertion of Corollary~\ref{cor:large-remainder}, applied to
\(\mathcal R_0\), gives
\[
 e(\mathcal F)
 \ge e(\mathcal E_1,\mathcal R_1)
 >2^{\lfloor n/2\rfloor-4}a_1.
\]
Moreover,
\[
 a_1=|\mathcal F|-|\mathcal R_1|
 \ge \frac38\,2^{\lfloor n/2\rfloor}+s.
\]
Since \(s\le 2^{\lfloor n/2\rfloor}/26\), the displayed lower bound for
\(a_1\) exceeds \(8s\).  Hence
\[
 e(\mathcal F)>2^{\lfloor n/2\rfloor-4}\cdot8s
 =s\,2^{\lfloor n/2\rfloor-1},
\]
a contradiction.  Therefore
\begin{equation}
 |\mathcal R_1|>\frac58\,2^{\lfloor n/2\rfloor}.
\label{eq:first-remainder-very-large}
\end{equation}

We now let \(i_0\ge1\) be the largest index such that
\[
 |\mathcal R_{i_0}|>\frac58\,2^{\lfloor n/2\rfloor}.
\]
The existence of \(i_0\) follows from \eqref{eq:first-remainder-very-large}, and
the process terminates because \(\mathcal F\) is finite.  We then have
\[
 |\mathcal R_i|>\frac58\,2^{\lfloor n/2\rfloor}
 \quad(1\le i\le i_0),
 \qquad
 |\mathcal R_{i_0+1}|
 \le
 \frac58\,2^{\lfloor n/2\rfloor}.
\]
For each \(1\le i\le i_0\), apply the second assertion of Corollary
\ref{cor:large-remainder} to
\(\mathcal R_{i-1}\), its maximum eventown subfamily
\(\mathcal E_i\), and the remainder \(\mathcal R_i\).  This gives
\[
 e(\mathcal E_i,\mathcal R_i)
 >
 \frac{3}{32}\,2^{\lfloor n/2\rfloor}a_i.
\]
Combining these inequalities with \eqref{eq:edge-partition}, we obtain
\begin{equation}
 e(\mathcal F)>\frac{3}{32}\,2^{\lfloor n/2\rfloor}
 \sum_{i=1}^{i_0}a_i.
\label{eq:edges-from-layers}
\end{equation}

It remains to estimate \(\sum_{i=1}^{i_0}a_i\).  By \eqref{eq:layer-monotonicity} and
\(i_0\ge1\), one has
\[
 a_{i_0+1}\le a_{i_0}\le \sum_{i=1}^{i_0}a_i.
\]
On the other hand, the first \(i_0+1\) layers contain every element of
\(\mathcal F\) outside \(\mathcal R_{i_0+1}\), and therefore
\[
 a_{i_0+1}+\sum_{i=1}^{i_0}a_i
 =|\mathcal F|-|\mathcal R_{i_0+1}|
 \ge \frac38\,2^{\lfloor n/2\rfloor}+s.
\]
Since \(a_{i_0+1}\le \sum_{i=1}^{i_0}a_i\), it follows that
\begin{equation}
 \sum_{i=1}^{i_0}a_i\ge
 \frac{a_{i_0+1}+\sum_{i=1}^{i_0}a_i}{2}\ge
 \frac{\frac38\,2^{\lfloor n/2\rfloor}+s}{2}.
\label{eq:layer-mass}
\end{equation}

The assumption \(s\le 2^{\lfloor n/2\rfloor}/26\) implies
\[
 \frac38\,2^{\lfloor n/2\rfloor}+s
 >\frac{32}{3}s.
\]
Thus \eqref{eq:layer-mass} shows
\(\sum_{i=1}^{i_0}a_i>\frac{16}{3}s\). Together with
\eqref{eq:edges-from-layers}, this yields
\[
 e(\mathcal F)
 >
 \frac{3}{32}\,2^{\lfloor n/2\rfloor}\cdot\frac{16}{3}s
 =
 s\,2^{\lfloor n/2\rfloor-1},
\]
contradicting the choice of \(\mathcal F\).
\end{proof}

\begin{remark}\label{rem:constants}
The constants \(5/8\), \(3/32\), and \(1/26\) are chosen for convenient arithmetic
rather than optimality.  The second assertion of
Corollary~\ref{cor:large-remainder} is obtained by moving the stopping
threshold above one half, while the first assertion is retained to handle the
intermediate range
\[
 2^{\lfloor n/2\rfloor-1}<|\mathcal R_1|
 \le \frac58\,2^{\lfloor n/2\rfloor}.
\]
This keeps the proof elementary and avoids introducing an additional optimization parameter.
\end{remark}

Theorem~\ref{thm:main} can also be used to give
the following lower bound.

\begin{proposition}\label{prop:averaging}
Let \(\mathcal F\subseteq\mathcal V\) satisfy
\(|\mathcal F|=2^{\lfloor n/2\rfloor}+s\) for an integer \(s\ge0\), and define
\(r=
 \min\left\{
 s, \lfloor\frac{2^{\lfloor n/2\rfloor}}{26}\rfloor
 \right\}\). Then
\begin{equation}
 e(\mathcal F)
 \ge
 r\,2^{\lfloor n/2\rfloor-1}
 \cdot
 \frac{
 (2^{\lfloor n/2\rfloor}+s)
 (2^{\lfloor n/2\rfloor}+s-1)
 }{
 (2^{\lfloor n/2\rfloor}+r)
 (2^{\lfloor n/2\rfloor}+r-1)
 }.
\label{eq:averaging-bound}
\end{equation}
Consequently,
\[
 e(\mathcal F)
 \ge
 \max\left\{
 \frac{s(2^{\lfloor n/2\rfloor}+s)}{4},
 r\,2^{\lfloor n/2\rfloor-1}
 \cdot
 \frac{
 (2^{\lfloor n/2\rfloor}+s)
 (2^{\lfloor n/2\rfloor}+s-1)
 }{
 (2^{\lfloor n/2\rfloor}+r)
 (2^{\lfloor n/2\rfloor}+r-1)
 }
 \right\}.
\]
\end{proposition}

\begin{proof}
If \(r=0\), the right-hand side of \eqref{eq:averaging-bound} is zero, so
that inequality is immediate.  Suppose that \(r\ge1\).  Consider every
subfamily \(\mathcal B\subseteq\mathcal F\) of size
\(2^{\lfloor n/2\rfloor}+r\).  The definition of \(r\) allows us to apply
Theorem~\ref{thm:main} to each \(\mathcal B\), giving
\[
 e(\mathcal B)\ge r\,2^{\lfloor n/2\rfloor-1}.
\]
Each edge of \(\mathcal F\) belongs to exactly
\[
 \binom{2^{\lfloor n/2\rfloor}+s-2}
       {2^{\lfloor n/2\rfloor}+r-2}
\]
of these subfamilies.  Double counting pairs
\((\mathcal B,\text{edge of }\mathcal B)\) gives
\[
%\begin{aligned}
 e(\mathcal F)
 \binom{2^{\lfloor n/2\rfloor}+s-2}
       {2^{\lfloor n/2\rfloor}+r-2}
 \ge
 \binom{2^{\lfloor n/2\rfloor}+s}
       {2^{\lfloor n/2\rfloor}+r} 
 \cdot r\,2^{\lfloor n/2\rfloor-1}.
%\end{aligned}
\]
Dividing by the first binomial coefficient proves
\eqref{eq:averaging-bound}.  The standard quadratic Fourier estimate
\[
 e(\mathcal F)
 \ge
 \frac{s\bigl(2^{\lfloor n/2\rfloor}+s\bigr)}{4}
\]
is proved in \cite{WeiZhaoZhangGe}; taking the maximum of the two bounds gives
the final assertion.
\end{proof}

\section{Stability and removal in the large-core regime}\label{sec:stability}

We next apply Lemma~\ref{lem:switching-fourier} to families of extremal size
that have a large intersection with a generator.

\begin{theorem}\label{thm:stability}
Let \(\mathcal F\subseteq\mathcal V\) satisfy
\(|\mathcal F|=2^{\lfloor n/2\rfloor}\), and let \(\mathcal W\) be a generator maximizing
\(|\mathcal F\cap\mathcal W|\).  If
\(
 |\mathcal F\cap\mathcal W|
 \ge
 2^{\lfloor n/2\rfloor-1}\), then
\begin{equation}
 |\mathcal F\setminus\mathcal W|
 \le
 \frac{2e(\mathcal F)}{|\mathcal F\cap\mathcal W|}
 +
 \sqrt{
       \frac{2^{\lfloor n/2\rfloor+1}e(\mathcal F)}
            {|\mathcal F\cap\mathcal W|}
      }.
\label{eq:stability-exact}
\end{equation}
In particular,
\(
 |\mathcal F\setminus\mathcal W|
 \le 6\sqrt{e(\mathcal F)}\) and 
 \(
 |\mathcal F\triangle\mathcal W|
 \le 12\sqrt{e(\mathcal F)}\).
Thus, deleting at most \(6\sqrt{e(\mathcal F)}\) members yields an
eventown family.
\end{theorem}

\begin{proof}
Set
\(\mathcal E=\mathcal F\cap\mathcal W\) and 
 \(\mathcal Y=\mathcal F\setminus\mathcal W\). 
By Lemma~\ref{lem:maximum-eventown}, \(\mathcal E\) is a maximum eventown
subfamily of \(\mathcal F\).  Since
\(e(\mathcal E,\mathcal Y)\le e(\mathcal F)\), inequality
\eqref{eq:remainder-bound} gives \eqref{eq:stability-exact}.

The lower bound on \(|\mathcal E|\) implies
\[
 |\mathcal Y|
 \le
 2^{2-\lfloor n/2\rfloor}e(\mathcal F)
 +2\sqrt{e(\mathcal F)}.
\]
Since \(|\mathcal F|=2^{\lfloor n/2\rfloor}\),
\[
 e(\mathcal F)
 \le
 \binom{2^{\lfloor n/2\rfloor}}{2}
 <
 2^{2\lfloor n/2\rfloor}.
\]
Hence
\[
 2^{2-\lfloor n/2\rfloor}e(\mathcal F)
 \le
 4\sqrt{e(\mathcal F)},
\]
which proves
\(|\mathcal Y|\le6\sqrt{e(\mathcal F)}\).  Since
\(|\mathcal F|=|\mathcal W|\), we have
\[
 |\mathcal F\triangle\mathcal W|=2|\mathcal Y|,
\]
and the symmetric-difference estimate follows.  Finally,
\(\mathcal F\cap\mathcal W\) is eventown, so deleting \(\mathcal Y\)
proves the removal statement.
\end{proof}

\begin{corollary}\label{cor:small-edge}
Let \(\mathcal F\subseteq\mathcal V\) satisfy
\(|\mathcal F|=2^{\lfloor n/2\rfloor}\) and \(e(\mathcal F)<2^{\lfloor n/2\rfloor-1}\). 
Then there is a generator \(\mathcal W\) such that
\[
 |\mathcal F\triangle\mathcal W|
 \le12\sqrt{e(\mathcal F)}.
\]
In particular, deleting at most \(6\sqrt{e(\mathcal F)}\) members of
\(\mathcal F\) leaves an eventown family.
\end{corollary}

\begin{proof}
Let \(\mathcal E\) be a maximum eventown subfamily of \(\mathcal F\).  It
is also maximal by inclusion, so every element of
\(\mathcal F\setminus\mathcal E\) has a neighbor in \(\mathcal E\).  If
\(|\mathcal E|\le2^{\lfloor n/2\rfloor-1}\), then
\[
 e(\mathcal F)
 \ge
 e(\mathcal E,\mathcal F\setminus\mathcal E)
 \ge
 |\mathcal F\setminus\mathcal E|
 \ge
 2^{\lfloor n/2\rfloor-1},
\]
contrary to the hypothesis.  Thus
\(|\mathcal E|>2^{\lfloor n/2\rfloor-1}\).  Choose a generator
\(\mathcal W\) containing \(\mathcal E\).  Lemma
\ref{lem:maximum-eventown} shows that \(\mathcal W\) maximizes
\(|\mathcal F\cap\mathcal W|\) and that
\(\mathcal F\cap\mathcal W=\mathcal E\).  Theorem
\ref{thm:stability} now applies.
\end{proof}

\begin{remark}
The large-core assumption in Theorem~\ref{thm:stability} is automatic under
the stronger hypothesis in Corollary~\ref{cor:small-edge}.  A full stability
theorem would require showing, in a substantially wider edge range, that a
family with few odd-intersection pairs has a large intersection with some
generator.
\end{remark}

\section{Concluding remarks}\label{sec:conclusion}

In this work, we proved the optimal bound for the supersaturation problem in a wide range. The proof of Theorem~\ref{thm:main} uses one-dimensional switches, and the
constant arises from the threshold in Corollary~\ref{cor:large-remainder} and the greedy layer decomposition. It would be interesting to determine whether
switching arguments involving several cosets outside the generator can extend
Theorem~\ref{thm:main} to a larger part of the range in
Conjecture~\ref{conj:oneill}.

\section*{Declarations}

During the development of this work, the authors used OpenAI GPT-5.5 as a research-assistance tool to explore possible improvements of the argument of Wei, Zhao, Zhang and Ge. In particular, GPT-5.5 was used at an intermediate stage to investigate whether their range \(s\le 2^{\lfloor n/8\rfloor}/n\) could be extended, and this exploration led to a preliminary bound of order \(2^{\lfloor n/6 \rfloor}\). Building on this intermediate observation, the authors further developed the argument by introducing a layered greedy decomposition together with Fourier-analytic and switching ideas, which ultimately led to the results presented in this paper. All mathematical statements, proofs, computations, and citations in the final manuscript were independently checked and verified by the authors, who take full responsibility for the content of the paper.

%\section*{Acknowledgements}
%The author thanks the authors of the recent works on eventown supersaturation for motivating this problem.

% Use \bibliographystyle{siamplain} with the official SIAM macro package.
\bibliographystyle{plain}
\bibliography{eventown_references}

\end{document}